\documentclass{louloupart1}
\usepackage[utf8]{inputenc}
\usepackage{mathtools}
\usepackage{graphicx}
\usepackage{tikz}
\usetikzlibrary{cd}

\title[Automorphisms of the Artin group of type $D_5$]{Automorphisms of the Artin group of type $D_5$}
\author[L Paris]{Luis Paris}
\givenname{Luis}
\surname{Paris}
\address{Luis Paris, Université Bourgogne Europe, CNRS, IMB, UMR 5584, 21000 Dijon, France}
\email{lparis@u-bourgogne.fr}
\author[I Soroko]{Ignat Soroko}
\givenname{Ignat}
\surname{Soroko}
\address{Ignat Soroko, Department of Mathematics, Southern Methodist University, Dallas, TX 75205, USA}
\email{isoroko@smu.edu}

\newtheorem{thm}{Theorem} 
\newtheorem{lem}[thm]{Lemma}
\newtheorem{prop}[thm]{Proposition}
\newtheorem{corl}[thm]{Corollary}

\theoremstyle{definition}

\newtheorem{rem}[thm]{Remark}

\newtheorem*{acknow}{Acknowledgments}

\def\N{\mathbb N} 
\def\Z{\mathbb Z}

\newcommand{\Ker}{\operatorname{Ker}}

\newcommand{\End}{\operatorname{End}}
\renewcommand{\Im}{\operatorname{Im}} 
\newcommand{\Aut}{\operatorname{Aut}}
\newcommand{\Inn}{\operatorname{Inn}}
\newcommand{\Out}{\operatorname{Out}}
\newcommand{\conj}{\operatorname{conj}}

\newcommand{\ov}{\overline}
\mathchardef\mhyphen="2D

\renewcommand{\le}{\leqslant}
\renewcommand{\ge}{\geqslant}

\begin{document}

\begin{abstract}
For the Artin group of type $D_5$, we determine its automorphism group and the automorphism group of its quotient by the center. This settles the only remaining case, $n=5$, in the classification of automorphisms of Artin groups of spherical type $D_n$.

\smallskip\noindent
{\bf AMS Subject Classification\ \ } 
Primary: 20F36, secondary: 20F28, 57K20.

\smallskip\noindent
{\bf Keywords\ \ } 
Artin groups of type $D_n$, automorphisms.

\end{abstract}

\maketitle


\section{Introduction}\label{sec1}

Let $S$ be a finite set.
A \emph{Coxeter matrix} over $S$ is a square matrix $M=(m_{s,t})_{s,t\in S}$ indexed by the elements of $S$, with coefficients in $\N \cup \{\infty\}$, such that $m_{s,s}=1$ for all $s \in S$, and $m_{s,t} = m_{t,s} \ge 2$ for all $s,t\in S$, $s\neq t$.
Such a matrix is usually represented by a labeled graph, $\Gamma$, called a \emph{Coxeter graph}, defined by the following data.
The set of vertices of $\Gamma$ is $S$.
Two vertices $s,t\in S$ are connected by an edge if $m_{s,t}\ge 3$, and this edge is labeled with $m_{s,t}$ if $m_{s,t} \ge 4$.

If $a,b$ are two letters and $m$ is an integer $\ge 2$, then we denote by $\Pi(a,b,m)$ the alternating word $aba\!\ldots$ of length $m$.
In other words, $\Pi(a,b,m) = (ab)^{\frac{m}{2}}$ if $m$ is even, and $\Pi(a,b,m) = (ab)^{\frac{m-1}{2}} a$ if $m$ is odd.
Let $\Gamma$ be a Coxeter graph and let $M=(m_{s,t})_{s,t\in S}$ be its Coxeter matrix.
With $\Gamma$ we associate a group, $A[\Gamma]$, called the \emph{Artin group} of $\Gamma$, defined by the presentation
\[
A[\Gamma] = \langle S \mid \Pi(s,t,m_{s,t}) = \Pi(t,s,m_{s,t}) \text{ for } s, t \in S\,,\ s \neq t\,,\ m_{s,t}\neq \infty \rangle\,.
\]
The \emph{Coxeter group} of $\Gamma$, denoted by $W[\Gamma]$, is the quotient of $A[\Gamma]$ by the relations $s^2 = 1$, $s \in S$. We say that $\Gamma$ is \emph{of spherical type} if $W[\Gamma]$ is finite.

The natural problem of determining automorphism groups of Artin groups has been solved only for certain classes: spherical types $A_n$~\cite{DyeGro1,ChaCri1}, $B_n$~\cite{ChaCri1}, $I_2(m)$~\cite{GiHoMeRa1,CriPar0}, $D_4$~\cite{Sorok1}, and $D_n$ for $n\ge 6$~\cite{CasPar1}; affine types $\widetilde A_n$ and $\widetilde C_n$~\cite{ChaCri1}; some $2$-dimensional Artin groups~\cite{Crisp1,AnCho1}; large-type free-of-infinity Artin groups~\cite{Vasko1}; some other classes of large-type Artin groups~\cite{BlMaVa1,HuOsVa1}; and right-angled Artin groups, for which a uniform description and finite presentation of automorphism groups is known~\cite{Serva1,Laure1,Day1}.

In this paper we fill in the missing case $n=5$ in the $D_n$ sequence, i.e.\ we classify automorphisms of the Artin group $A[D_5]$ with the Coxeter graph given in Figure~\ref{fig:d5a4} on the left.

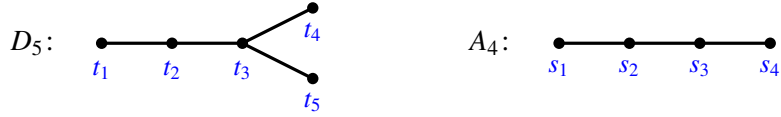
\begin{figure}
\begin{center}
\begin{tikzpicture}[scale=0.7, very thick]
\begin{scope}[scale=1.33] 
\fill (0,0) circle (2.3pt) node [below=2pt,blue] {\small$t_1$}; 
\fill (1,0) circle (2.3pt) node [below=2pt,blue] {\small$t_2$}; 
\fill (2,0) circle (2.3pt) node [below=2pt,blue] {\small$t_3$}; 
\fill (3,0.5) circle (2.3pt) node [below=1pt,blue] {\small$t_4$}; 
\fill (3,-0.5) circle (2.3pt) node [below=1pt,blue] {\small$t_5$}; 
\draw (3,0.5)--(2,0)--(3,-0.5);
\draw (0,0)--(1,0)--(2,0);
\draw (-1,0) node {$D_5$:};	
\end{scope}

\begin{scope}[scale=1.33, xshift=6.5cm] 
\fill (0,0) circle (2.3pt) node [below=2pt,blue] {\small$s_1$}; 
\fill (1,0) circle (2.3pt) node [below=2pt,blue] {\small$s_2$}; 
\fill (2,0) circle (2.3pt) node [below=2pt,blue] {\small$s_3$}; 
\fill (3,0) circle (2.3pt) node [below=2pt,blue] {\small$s_4$};
\draw (0,0)--(1,0)--(2,0)--(3,0);
\draw (-1,0) node {$A_4$:};	
\end{scope}
\end{tikzpicture}
\caption{The Coxeter graphs of types $D_5$ and $A_4$.\label{fig:d5a4}}
\end{center}
\end{figure}

Let $t_1,\dots,t_5$ denote the standard generators of $A[D_5]$, and let $\chi$ be the inversion automorphism of $A[D_5]$, defined on the generators as $\chi(t_i)=t_i^{-1}$, for $i=1,\dots,5$. Let $Z(G)$ denote the center of a group $G$. We prove the following theorems.
\begin{thm}\label{thm:1}
Let $A=A[D_5]$ and $C_2=\Z/2\Z$. Then
\[
\Aut(A)=\Inn(A)\rtimes\langle\chi\rangle\simeq A/Z(A)\rtimes C_2.
\]
In particular, $\Out(A)\simeq C_2$.
\end{thm}

\begin{thm}\label{thm:2}
Let $A=A[D_5]$, $\ov A=A/Z(A)$, and $\ov\chi$ be the
automorphism induced by $\chi$. Then
\[
\Aut(\ov A)=\Inn(\ov A)\rtimes
\langle\ov\chi\rangle\simeq A/Z(A)\rtimes C_2.
\]
In particular, $\Out(\ov A)\simeq C_2$.
\end{thm}

Thus the descriptions of $\Aut(A[D_5])$ and of $\Aut(A[D_5]/Z(A[D_5]))$ have the same form as the corresponding descriptions obtained by Castel and Paris for $A[D_n]$ with odd $n\ge6$; see \cite[Corollary~2.6\,(2)]{CasPar1}.

This paper is a continuation of the series of articles \cite{CasPar1,ParSor1,ParSor2} devoted to classification of automorphisms and endomorphisms of Artin groups that admit geometric representations as subgroups of mapping class groups of surfaces. The key ingredients of the proofs are descriptions of homomorphisms from braid groups to themselves and to mapping class groups, obtained in a series of earlier works~\cite{Caste0,Caste1,ChKoMa1,Orevk1}. 
 
Our paper is organized as follows. In Section~\ref{sec2} we prove that the kernel of the natural quotient map $A[D_5]\to A[A_4]$ is fully characteristic, and use the Crisp--Paris description of automorphisms preserving this kernel to prove Theorem~\ref{thm:1}. In Section~\ref{sec3} we prove Theorem~\ref{thm:2}. The main point is to show that a lift of an automorphism of $A[D_5]/Z(A[D_5])$ can be modified by a central transvection so as to become an automorphism of $A[D_5]$. Finally, in Section~\ref{sec4} we apply Theorem~\ref{thm:1} to fill the missing type $D_5$ case in a result of Jones--Mangioni--Sartori~\cite{JoMaSa1} on the strong twist conjecture.

\begin{acknow}
This work was started at the AIM workshop ``Geometry and topology of Artin groups'' organized by Ruth Charney, Kasia Jankiewicz, and Kevin Schreve in September 2023, and continued at the SLMath workshop ``Hot Topics: Artin Groups and Arrangements -- Topology, Geometry, and Combinatorics'' organized by Christin Bibby, Ruth Charney, Giovanni Paolini, and Mario Salvetti in March 2024, and during the second author's visit to Universit\'e Bourgogne Europe in May--June 2025. We thank the organizers, AIM, SLMath and Universit\'e Bourgogne Europe for hospitality.
The first author is partially supported by the French project ``CaGeT'' (ANR-25-CE40-4162) of the ANR.
The IMB receives support from the EIPHI Graduate School (contract ANR-17-EURE-0002).
\end{acknow}


\section{\texorpdfstring{Automorphisms of $A[D_5]$}{Automorphisms of A[D5]}}\label{sec2}
The Artin group $A[A_4]$ will play an important auxiliary role in our work. Its Coxeter graph $A_4$ is shown on the right
in Figure~\ref{fig:d5a4}. It appears as the image of the quotient map from $A[D_5]$ obtained by identifying the two terminal generators $t_4$ and $t_5$ of the Coxeter graph $D_5$, see the epimorphism $\pi$ below. We denote the standard generators of $A[A_4]$ by $s_1,\dots,s_4$. 

Let $\Delta$ be the Garside element of $A[A_4]$, and let $\Delta_D$ be the Garside element of $A[D_5]$. It is known that $Z(A[A_4])=\langle \Delta^2\rangle$ and $Z(A[D_5])=\langle \Delta_D^2\rangle$; see~\cite[Satz~7.2]{BriSai1}. We will use the following expression of the element $\Delta^2$ as a product of standard generators: $\Delta^2=(s_1s_2s_3s_4)^5$, see for example~\cite[\S5.8, Zusatz~(i)]{BriSai1}.
Moreover, the conjugation by $\Delta_D$ induces the nontrivial graph automorphism of $A[D_5]$, namely
\[
\Delta_D t_i\Delta_D^{-1}=t_i\quad(1\le i\le 3),\qquad
\Delta_D t_4\Delta_D^{-1}=t_5,\qquad
\Delta_D t_5\Delta_D^{-1}=t_4.
\]
We now define a few useful homomorphisms, following the notation of~\cite{CasPar1}.

\underline{$\pi\colon A[D_5]\to A[A_4]$}: \quad$\pi(t_i)=s_i \quad (1\le i\le 3), \qquad \pi(t_4)=\pi(t_5)=s_4$.

It is known that $\pi(\Delta_D)=\Delta^2$, see~\cite[Lemma~5.1 and the proof of Theorem~2.8]{CasPar1}.

\underline{$\alpha_p\colon A[D_5]\to A[A_4]$}:\quad For $p\in\Z$ define
\[
\alpha_p(t_i)=s_i\Delta^{2p}\quad (1\le i\le 3),
\qquad
\alpha_p(t_4)=\alpha_p(t_5)=s_4\Delta^{2p}.
\]
Since $\Delta^2$ is central in $A[A_4]$, $\alpha_p$ preserves the defining relations of $A[D_5]$ and hence indeed defines a homomorphism $A[D_5]\to A[A_4]$. Observe that $\alpha_0=\pi$.

\underline{$\gamma_p\colon A[A_4]\to A[A_4]$}:\quad For $p\in\Z$ define
\[
\gamma_p(s_i)=s_i\Delta^{2p}\qquad (1\le i\le 4).
\]
By the same argument as above, $\gamma_p$ is a homomorphism. Observe that $\alpha_p=\gamma_p\circ \pi$.

\underline{$\chi\colon A[D_5]\to A[D_5]$}:\quad Recall that $\chi$ was defined as:
\[
\chi(t_i)=t_i^{-1},\qquad (1\le i\le 5).
\]
\underline{$\chi_A\colon A[A_4]\to A[A_4]$}:\quad We define $\chi_A$ in a similar way:
\[
\chi_A(s_i)=s_i^{-1},\qquad (1\le i\le 4).
\]

For the rest of the paper, denote
\[
F:=\Ker\pi.
\]
By Crisp--Paris~\cite[Proposition~2.3\,(2)]{CriPar1}, $F$ is a free group of rank $4$, and $A[D_5]\simeq F\rtimes A[A_4]$, where the embedding $A[A_4]\hookrightarrow A[D_5]$ is given by $s_i\mapsto t_i$, $1\le i\le4$.

If $G$ is a group and $g\in G$, then we denote by $\conj_g\colon G\to G$, $h\mapsto ghg^{-1}$, the conjugation by $g$. We say that two homomorphisms $\varphi_1,\varphi_2\colon G\to H$ are \emph{conjugate} if there exists $h\in H$ such that $\varphi_2=\conj_h\circ\varphi_1$. A homomorphism $\varphi\colon G\to H$ is called \emph{cyclic} if its image is a cyclic subgroup of $H$. If $G$ is either $A[A_4]$ or $A[D_5]$, then a homomorphism $\varphi\colon G\to H$ is cyclic if and only if there exists $h\in H$
such that all standard generators of $G$ are sent by $\varphi$ to $h$.

\begin{lem}\label{lem:surj}
Let $\varphi\colon A[D_5]\to A[A_4]$ be a surjective homomorphism. Then $\varphi$ is conjugate either to $\pi$ or to $\pi\circ \chi$. In particular, $\Ker(\varphi)=F$.
\end{lem}
\begin{proof}
By~\cite[Theorem~2.1]{CasPar1}, an arbitrary homomorphism $\varphi\colon A[D_5]\to A[A_4]$ has one of the two forms: 
\begin{enumerate}
\item either $\varphi$ is cyclic, or
\item there exist $\psi\in\langle\chi\rangle$ and $p\in\Z$ such that $\varphi$ is conjugate to $\alpha_p\circ\psi$.
\end{enumerate}
Since $\varphi$ is surjective, case (1) is impossible, so $\varphi$ is conjugate to $\alpha_p\circ\psi$ for some $p\in\Z$ and $\psi\in\langle\chi\rangle$. Since $\psi$ is an automorphism of $A[D_5]$, it does not affect surjectivity of $\varphi$, hence it is enough to analyze when $\alpha_p$ is surjective. Arguing as in the proof of~\cite[Corollary~2.5]{CasPar1}, define the exponent sum homomorphism
$\ell\colon A[A_4]\to\Z$ which sends $s_i\mapsto1$ for all $i$. It is known that $\Delta^2=(s_1s_2s_3s_4)^5$ and hence $\ell(\Delta^2)=20$. In particular, $\ell(\alpha_p(t_i))=\ell(s_j\Delta^{2p})=1+20p$, for all $i,j$. Therefore, $\Im(\ell\circ\alpha_p)\subseteq (1+20p)\Z$. Since $\alpha_p$ is assumed to be surjective, this is only possible if $p=0$. Hence $\varphi$ is conjugate to $\alpha_0\circ\psi$, which is equal to either $\pi$ or $\pi\circ\chi$. Since $\pi\circ\chi=\chi_A\circ\pi$, and since conjugation in the codomain does not change the kernel, we conclude that $\Ker\varphi=\Ker\pi=F$.
\end{proof}

Recall that a subgroup $N$ of a group $G$ is called \emph{characteristic} if for every automorphism $\Phi\in\Aut(G)$ we have $\Phi(N)=N$, and $N$ is called \emph{fully characteristic} if for every endomorphism $\varphi\in\End(G)$ we have $\varphi(N)\subseteq N$.

\begin{lem}\label{lem:F-fully-char}
Let $\pi\colon A[D_5]\to A[A_4]$ 
be the epimorphism defined above, and let $F=\Ker\pi$. Then $F$ is fully
characteristic in $A[D_5]$.
\end{lem}

\begin{proof}
Let $\varphi\colon A[D_5]\to A[D_5]$ be an arbitrary endomorphism, and consider $\lambda=\pi\circ\varphi:A[D_5]\to A[A_4]$. By \cite[Theorem~2.1]{CasPar1}, either $\lambda$ is cyclic, or $\lambda$ is conjugate to $\alpha_p\circ\psi$ for some $p\in\Z$ and $\psi\in\langle\chi\rangle$.

Suppose first that $\lambda$ is not cyclic. Write $\alpha_p=\gamma_p\circ\pi$, and observe that $\gamma_p$ is injective. Indeed, consider again the exponent-sum homomorphism $\ell\colon A[A_4]\to\Z$, which sends $s_i\mapsto1$ for all $i$. Since $\Delta^2$ is central, for any $g\in A[A_4]$ we have $\gamma_p(g)=g\Delta^{2p\,\ell(g)}$. Suppose that $\gamma_p(g)=1$. Then $g=\Delta^{-2p\,\ell(g)}$. Applying $\ell$ and using $\ell(\Delta^2)=20$, we get $\ell(g)=-20p\,\ell(g)$. Hence $(1+20p)\,\ell(g)=0$, so $\ell(g)=0$. Hence $g=\Delta^{-2p\cdot0}=1$, and $\gamma_p$ is injective.

We have: $\lambda$ is conjugate to $\gamma_p\circ\pi\circ\psi$, $\gamma_p$ is injective, $\psi(F)=F$ (since $\psi\in\langle\chi\rangle$ and $\pi\circ\chi=\chi_A\circ\pi$), and conjugating the image does not change the kernel. Therefore $\Ker\lambda=\Ker\pi=F$. On the other hand, $\Ker\lambda=\Ker(\pi\circ\varphi)=\varphi^{-1}(F)$. Hence $\varphi^{-1}(F)=F$, and, in particular, $\varphi(F)\subseteq F$.

Suppose now that $\lambda$ is cyclic. Then there exists $h\in A[A_4]$ such that $\lambda(t_i)=h$ for all $i$. In particular, $\lambda(t_5t_4^{-1})=1$. The subgroup $F=\Ker\pi$ is the normal closure of $t_5t_4^{-1}$ in $A[D_5]$, since $\pi$ is the quotient map obtained by identifying $t_4$ and $t_5$. Therefore $\lambda(F)=1$, i.e.\ $\pi(\varphi(F))=1$, so $\varphi(F)\subseteq F$.

Thus every endomorphism of $A[D_5]$ preserves $F$.
\end{proof}

\begin{proof}[Proof of Theorem~\ref{thm:1}]
Crisp and Paris proved in~\cite[Theorem~4.9]{CriPar1} that the subgroup $\Aut(A[D_5],F)$ consisting of all automorphisms of $A[D_5]$ preserving $F$, has the form:
\[
\Aut(A[D_5],F)=\Inn(A[D_5])\rtimes\langle\chi\rangle\simeq A[D_5]/Z(A[D_5])\rtimes C_2.
\]
It follows from Lemma~\ref{lem:F-fully-char} that the subgroup $F$ is characteristic, hence $\Aut(A[D_5],F)=\Aut(A[D_5])$, and the theorem follows.
\end{proof}

\begin{rem}\label{rem:endos}
In addition to automorphisms, Castel and Paris also give a classification of all
endomorphisms of $A[D_n]$ for $n\ge6$.  Their proof combines two ingredients:
\begin{itemize}
\item a classification of homomorphisms $A[D_n]\to A[A_{n-1}]$,~\cite[Theorem~2.1]{CasPar1}, 
\item a classification of homomorphisms $A[A_{n-1}]\to A[D_n]$,~\cite[Theorem~2.2]{CasPar1}.
\end{itemize}
The first ingredient is still available for $D_5$, since~\cite[Theorem~2.1]{CasPar1} applies for $n\ge5$. It is based on a classification of endomorphisms of $A[A_{n-1}]$, which is given for $n=5$ independently in~\cite{ChKoMa1} and in~\cite{Orevk1}. However, the second ingredient, Theorem~2.2 of Castel--Paris, essentially uses Castel's classification of maps from $A[A_{n-1}]$ to the mapping class group of the surface of genus $\lfloor(n-1)/2\rfloor$ with one or two boundary components, after applying the Perron--Vannier embedding of $A[D_n]$. This classification is available to us only for $n\ge 6$, via the earlier work of Castel~\cite[Theorem~1]{Caste1}.  To describe endomorphisms of $A[D_5]$ by the same route, one would need the genus $2$ analogue of Castel's result: namely, the statement that every non-cyclic homomorphism from $A[A_4]$ to the mapping class group of a genus $2$ surface with one boundary component is obtained from the Birman--Hilden $4$-chain Dehn twist representation by conjugating with a homeomorphism, applying a central transvection, and possibly precomposing with the inversion automorphism of $A[A_4]$. We do not have a proof of such a statement at present. However, \emph{surjective} endomorphisms of $A[D_5]$ admit a complete description and are given in Theorem~\ref{thm:1}, since all of them are automorphisms. Indeed, by the results of Digne~\cite{Digne1} and Cohen--Wales~\cite{CohWal1}, Artin groups of spherical type are linear, and, being finitely generated, they are Hopfian by Malcev's theorem (which means that every surjective endomorphism is an automorphism).
\end{rem}

\section{\texorpdfstring{Automorphisms of $A[D_5]/Z(A[D_5])$}{Automorphisms of A[D5]/Z(A[D5])}}\label{sec3}

Denote for brevity $\ov{A[D_5]}=A[D_5]/Z(A[D_5])$ and let $\xi\colon A[D_5]\to A[D_5]/Z(A[D_5])$ be the canonical projection. Note that if an endomorphism $\varphi\colon A[D_5]\to A[D_5]$ has the property $\varphi(Z(A[D_5]))\subseteq Z(A[D_5])$, it induces a natural endomorphism $\ov\varphi\colon\ov{A[D_5]}\to \ov{A[D_5]}$. Let $\phi\colon \ov{A[D_5]}\to \ov{A[D_5]}$ be an arbitrary endomorphism. If there exists an endomorphism $\varphi\colon A[D_5]\to A[D_5]$ making the following diagram commutative
\begin{center}
\setlength\mathsurround{0pt}
\begin{tikzcd}
A[D_5] \arrow[r,"\varphi"] \arrow[d,"\xi"] & A[D_5]\arrow[d,"\xi"]\\
\ov{A[D_5]}\arrow[r,"\phi"] & \ov{A[D_5]}
\end{tikzcd}
\end{center}
we say that $\phi$ \emph{lifts} and that $\varphi$ is a \emph{lift} of $\phi$. Notice that any lift $\varphi$ of any endomorphism $\phi$ of $\ov{A[D_5]}$ must have property $\varphi(Z(A[D_5]))\subseteq Z(A[D_5])$.

Now let $\phi\in\Aut(\ov{A[D_5]})$ be an arbitrary automorphism. By~\cite[Proposition~2.7]{CasPar1}, $\phi$ admits a lift
$\varphi\colon A[D_5]\to A[D_5]$, which is a priori only an endomorphism. We will show that this lift can be modified, without changing the induced automorphism of $\ov{A[D_5]}$, so that it becomes an automorphism of $A[D_5]$.

Let $\varphi\colon A[D_5]\to A[D_5]$ be a fixed lift of $\phi$. Recall that by~\cite[Proposition~2.3\,(2)]{CriPar1} we have
\[
A[D_5]\simeq F\rtimes A[A_4],
\]
where $F=\Ker\pi$ is a free group of rank $4$, and $A[A_4]=\Im\pi$ is identified with the subgroup of $A[D_5]$ generated by $t_1,\dots,t_4$. By Lemma~\ref{lem:F-fully-char}, the subgroup $F$ is fully characteristic in $A[D_5]$. Hence $\varphi(F)\subseteq F$, and $\varphi$ induces an endomorphism
\[
\varphi_A\colon A[A_4]\to A[A_4]
\] 
of $A[D_5]/F=A[A_4]$ by $\varphi_A(\pi(x)):=\pi(\varphi(x))$, for all $x\in A[D_5]$.

We claim that $\varphi_A$ is non-cyclic. Since $\phi$ is surjective on $\ov{A[D_5]}$, we have
$A[D_5]=\varphi(A[D_5])\,Z(A[D_5])$. Applying $\pi$ gives $A[A_4]=\pi(A[D_5])=\pi(\varphi(A[D_5]))\,\pi(Z(A[D_5]))$.
The subgroup $\pi(Z(A[D_5]))$ is central in $A[A_4]$. Thus, if $\pi(\varphi(A[D_5]))=\varphi_A(A[A_4])$ were cyclic, then $A[A_4]$ would be abelian, a contradiction. Therefore $\varphi_A$ is non-cyclic.

By~\cite[Corollary~1.2]{ChKoMa1}, or by~\cite[Theorem~1.6]{Orevk1}, there exist $g\in A[A_4]$, $\varepsilon\in\{0,1\}$, and $p\in\Z$ such that
\begin{equation}\label{eq:normal-form}
\varphi_A=\conj_g\circ\gamma_p\circ\chi_A^\varepsilon\,,
\end{equation}
where $\gamma_p(s_i)=s_i\Delta^{2p}$ for $1\le i\le4$. The next lemma is our key observation that the exponent $p$ in this representation must be even.

\begin{lem}\label{lem:p-even}
Let $\phi\in\Aut(\ov{A[D_5]})$, and let $\varphi$, $\varphi_A$, and $p$ be
defined as above. Then $p$ is even.
\end{lem}

\begin{proof}
We prove the lemma in three steps. First, we use the fact that $\phi$ is an automorphism of $\ov{A[D_5]}$ to show that the restriction
$\theta=\varphi|_F$ is an automorphism of $F$. Second, we use the semidirect product structure $A[D_5]\simeq F\rtimes A[A_4]$ to compare, on the abelianization $F_{ab}$ of $F$, the action of an element $a\in A[A_4]$ with the action of its image $\varphi_A(a)$. Third, we apply the explicit Crisp--Paris formulas to this comparison: it turns out that the standard generators $s_i$ act with the characteristic polynomial $(X-1)^4$, while an odd value of $p$ would force $\varphi_A(s_i)$ to have characteristic polynomial $(X+1)^4$ on $F_{ab}$, which yields a contradiction.

Let $\theta=\varphi|_F$ and denote $\overline F=\xi(F)$. We show that $\theta$ is an automorphism of $F$. Since $F$ is fully characteristic by Lemma~\ref{lem:F-fully-char}, and since $\varphi$ is a lift of $\phi$, we have $\phi(\xi(f))=\xi(\varphi(f))\in \xi(F)$ for every $f\in F$. Thus $\phi(\overline F)\subseteq \overline F$. Applying the same argument to $\phi^{-1}$ (using some lift of $\phi^{-1}$), we get $\phi^{-1}(\overline F)\subseteq \overline F$, hence $\phi(\overline F)=\overline F$. Since $F\cap Z(A[D_5])=1$, the restriction $\xi|_F\colon F\to\overline F$ is an isomorphism. Moreover, for every $f\in F$ we have $\xi(\theta(f))=\xi(\varphi(f))=\phi(\xi(f))$. Therefore $\theta=(\xi|_F)^{-1}\circ\phi|_{\overline F}\circ\xi|_F$, and so $\theta$ is an automorphism of $F$.

Let $\rho\colon A[A_4]\to\Aut(F)$ be the homomorphism underlying the semidirect product $A[D_5]\simeq F\rtimes A[A_4]$, i.e.\ for any $a\in A[A_4]$ and any $x\in F$, we have 
\[
a\,x\,a^{-1}=\rho(a)(x).
\]
Write the image of $a$ under $\varphi$ in the semidirect product form as
\[
\varphi(a)=f_a\cdot\varphi_A(a),\qquad \text{with\quad } f_a\in F.
\]
Applying $\varphi$ to the identity $a\,x\,a^{-1}=\rho(a)(x)$ gives
\[
\varphi(a)\,\varphi(x)\,\varphi(a)^{-1} = \varphi(\rho(a)(x)).
\]
Since $x$ and $\rho(a)(x)$ lie in $F$, and $\theta=\varphi|_F$, this becomes
\[
f_a\,\varphi_A(a)\,\theta(x)\,\varphi_A(a)^{-1}\,f_a^{-1} = \theta(\rho(a)(x)).
\]
Equivalently,
\[
\conj_{f_a}\bigl(\rho(\varphi_A(a))(\theta(x))\bigr)
=
\theta(\rho(a)(x)).
\]
Since this holds for every $x\in F$, we get an equality in $\Aut(F)$:
\[
\theta\circ\rho(a)=\conj_{f_a}\circ\,\rho(\varphi_A(a))\circ\theta,
\]
which is equivalent to
\[
\theta\,\circ\rho(a)\,\circ\theta^{-1} = \conj_{f_a}\circ\,\rho(\varphi_A(a)).
\]
Denote by $\theta_*$ and $\rho_*(a)$ the automorphisms of $F_{ab}=F/[F,F]$ induced by $\theta$ and $\rho(a)$, respectively. Since the inner automorphisms $\conj_{f_a}$ act trivially on $F_{ab}$, we get
\begin{equation}\label{eq:theta-star}
\theta_*\,\rho_*(a)\,\theta_*^{-1} = \rho_*(\varphi_A(a))\,,
\end{equation}
for all $a\in A[A_4]$.

Now we are going to apply explicit formulas for the action of $A[A_4]$ on $F$, which were computed in Crisp--Paris' paper~\cite[Proposition~2.3]{CriPar1} for arbitrary $n$.

Crisp--Paris use a different notation for the standard generators of $A[D_n]$ and $A[A_{n-1}]$, which they denote $\delta_1,\dots,\delta_n$, and $\alpha_1,\dots,\alpha_{n-1}$, respectively. We use their numbering after interchanging the two vertices of the fork in the Coxeter graph of $D_5$, and identify them with our generators for $n=5$ by
\[
\delta_i=t_{6-i},\quad (1\le i\le5); \qquad s_i=\alpha_{5-i}\qquad (1\le i\le4).
\]
With this convention, the Crisp--Paris section $\alpha_i\mapsto \delta_{i+1}$ identifies $A[A_4]$ with our chosen subgroup $\langle t_1,t_2,t_3,t_4\rangle\le A[D_5]$.

By the formulas for the isomorphism $\psi\colon F\rtimes A[A_4]\to A[D_5]$ in the proof of \cite[Proposition~2.3]{CriPar1},
translated to our notation, $F$ has free basis $x_1,x_2,x_3,x_4$, where
\[
x_1=t_5t_4^{-1},\qquad
x_2=t_3\cdot x_1\cdot t_3^{-1},\qquad
x_3=t_2t_3\cdot x_1\cdot t_3^{-1}t_2^{-1},\qquad
x_4=t_1t_2t_3\cdot x_1\cdot t_3^{-1}t_2^{-1}t_1^{-1},
\]
and the action $\rho\colon A[A_4]\to\Aut(F)$ is given by
\[
\rho(s_{4-r})(x_j)=
\begin{cases}
x_{r+1}, &  j=r,\\
x_{r+1}x_r^{-1}x_{r+1}, &  j=r+1,\\
x_j, & j\notin\{r,r+1\},
\end{cases}
\ \ \text{for\ \ } 1\le r\le 3,\quad \text{and\quad }
\rho(s_4)(x_j)=
\begin{cases}
x_1, & j=1,\\
x_1^{-1}x_j, & j\ne 1.
\end{cases}
\]
If we denote by $\rho_*\colon A[A_4]\to \Aut(F_{ab})$ the induced representation on $F_{ab}$, it follows that in the basis $[x_1],\dots,[x_4]$ of $F_{ab}$, we have:
\[
\rho_*(s_4)=
\left[\begin{array}{r@{\;\;}r@{\;\;}r@{\;\;}r} 
1 & -1 & -1 & -1\\
0 & 1 & 0 & 0\\
0 & 0 & 1 & 0\\
0 & 0 & 0 & 1
\end{array}\right],
\qquad
\text{and\quad $\rho_*(s_i)$\ $(1\le i\le3)$\quad has blocks:\quad }
\left[\begin{array}{r@{\;\;}r} 
0 & -1\\
1 & 2
\end{array}\right],\,[1],\,[1].
\]
An easy computation shows that every $\rho_*(s_i)$ has characteristic polynomial $(X-1)^4$.

Since $\pi(\Delta_D)=\Delta^2$ by \cite[proof of Theorem~2.8]{CasPar1}, and since we view $A[A_4]$ as the subgroup $\langle t_1,t_2,t_3,t_4\rangle$ in the splitting $A[D_5]\simeq F\rtimes A[A_4]$, the elements $\Delta_D$ and $\Delta^2$ differ by an element of $F$. Thus conjugation by $\Delta_D$ and the action $\rho(\Delta^2)$ differ on $F$ by an inner automorphism of $F$, which is trivial on $F_{ab}$.

Now conjugation by $\Delta_D$ fixes $t_1,t_2,t_3$ and swaps $t_4,t_5$. Hence it sends each generator $x_j$ above to $x_j^{-1}$. Therefore $\rho_*(\Delta^2)$ sends each $[x_j]$ to $-[x_j]$, and so $\rho_*(\Delta^2)=-I$ on $F_{ab}$.

Recall that $\varphi_A$ has the form given in~\eqref{eq:normal-form}, $\varphi_A=\conj_g\circ\gamma_p\circ\chi_A^\varepsilon$,
which means that for all $1\le i\le4$ we have: $\varphi_A(s_i)=g\,\gamma_p(\chi_A^\varepsilon(s_i))\,g^{-1}$.
Applying $\rho_*$, we observe that the conjugation by $g$ changes the matrix only by similarity, and hence does not change its characteristic polynomial.

Assume now that $p$ is odd. 

If $\varepsilon=0$, then $\rho_*(\gamma_p(s_i)) = \rho_*(s_i\Delta^{2p}) = \rho_*(s_i)\rho_*(\Delta^2)^p = \rho_*(s_i)\,(-I)^p= -\rho_*(s_i)$.

If $\varepsilon=1$, then $\chi_A(s_i)=s_i^{-1}$, and therefore
$\gamma_p(\chi_A(s_i))=(s_i\Delta^{2p})^{-1}$. Hence
\[
\rho_*(\gamma_p(\chi_A(s_i))) = \rho_*((s_i\Delta^{2p})^{-1}) = \rho_*(s_i)^{-1}\,(-I)^{-p} = -\rho_*(s_i)^{-1}.
\]
In both cases $\rho_*(\varphi_A(s_i))$ has characteristic polynomial $(X+1)^4$.

On the other hand, applying the identity~\eqref{eq:theta-star}, $\theta_*\rho_*(a)\,\theta_*^{-1}=\rho_*(\varphi_A(a))$, to $a=s_i$, we see that $\rho_*(\varphi_A(s_i))$ is similar to $\rho_*(s_i)$. Hence it has characteristic polynomial $(X-1)^4$, a contradiction. Therefore $p$ must be even.
\end{proof}

\begin{proof}[Proof of Theorem~\ref{thm:2}]
Let $\phi\in\Aut(\ov{A[D_5]})$. Arguing as in the beginning of this section, choose a lift $\varphi\colon A[D_5]\to A[D_5]$. As above, $\varphi$ preserves $F$ and induces a non-cyclic endomorphism $\varphi_A$ of $A[A_4]$. Hence, by the classification of non-cyclic endomorphisms of $A[A_4]$, we have
\[
\varphi_A=\conj_g\circ\gamma_p\circ\chi_A^\varepsilon,
\]
where $g\in A[A_4]$, $\varepsilon\in\{0,1\}$, and $\gamma_p(s_i)=s_i\Delta^{2p}$ for all $i$. By Lemma~\ref{lem:p-even}, $p$ is even.

Let $z=\Delta_D^2$, the central generator of $A[D_5]$. By the computation in the proof of \cite[Theorem~2.8]{CasPar1}, $\pi(\Delta_D)=\Delta^2$,
hence $\pi(z)=\Delta^4$. Write $p=2q$. Let $\ell\colon A[D_5]\to\Z$, sending $t_i\mapsto 1$ for all $i$, be the exponent-sum homomorphism, and define a new homomorphism $\varphi'\colon A[D_5]\to A[D_5]$ by
\[
\varphi'(x)=\varphi(x)\,z^{m\,\ell(x)},
\]
for all $x\in A[D_5]$, where $m=-q$ if $\varepsilon=0$, and $m=q$ if $\varepsilon=1$. Since $z$ is central, $\varphi'$ is indeed a homomorphism. Also $\varphi'$ induces the same automorphism $\phi$ of $\ov{A[D_5]}$, because $z\in Z(A[D_5])$.

Let us compute the induced map $\varphi'_A$ on $A[D_5]/F\simeq A[A_4]$. For $1\le i\le4$, the generator $s_i$ is represented by $t_i$, and therefore
\[
\varphi'_A(s_i)=\pi(\varphi'(t_i))=\pi(\varphi(t_i)\,z^m)=\pi(\varphi(t_i))\,\pi(z)^m =\varphi_A(s_i)\,\Delta^{4m}.
\]
If $\varepsilon=0$, then $\varphi_A(s_i)=g s_i g^{-1}\Delta^{2p}$, so $\varphi'_A(s_i)=g s_i g^{-1}\Delta^{2p+4m}=g s_i g^{-1}$. If
$\varepsilon=1$, then $\varphi_A(s_i)=g s_i^{-1}g^{-1}\Delta^{-2p}$, so $\varphi'_A(s_i)=g s_i^{-1}g^{-1}\Delta^{-2p+4m}=g s_i^{-1}g^{-1}$.
Thus in both cases $\varphi'_A=\conj_g\circ\chi_A^\varepsilon$, and hence $\varphi'_A$ is an automorphism of $A[A_4]$.

Again, arguing as in the proof of Lemma~\ref{lem:p-even}, we see that $\varphi'$ preserves $F$, restricts to an automorphism of $F$, and induces an automorphism of $A[D_5]/F$. Thus, by the five-lemma argument for group extensions, $\varphi'$ is an automorphism of $A[D_5]$ itself.

Now, by Theorem~\ref{thm:1}, every automorphism of $A[D_5]$ belongs to $\Inn(A[D_5])\rtimes\langle\chi\rangle$. Since $\varphi'$ induces $\phi$
on $\ov{A[D_5]}$, it follows that $\phi\in \Inn(\ov{A[D_5]})\,\langle\ov\chi\rangle$, which shows that $\Aut(\ov{A[D_5]})\subseteq\Inn(\ov{A[D_5]})\,\langle\ov\chi\rangle$. The reverse inclusion is obvious, so 
\[
\Aut(\ov{A[D_5]})=\Inn(\ov{A[D_5]})\,\langle\ov\chi\rangle.
\]
By~\cite[Lemma~7.1]{CasPar1}, $\ov\chi\notin\Inn(\ov{A[D_5]})$, so this product is semidirect. Finally, the center of $\ov{A[D_5]}$ is trivial, hence $\Inn(\ov{A[D_5]})\simeq \ov{A[D_5]}$ and $\Out(\ov{A[D_5]})\simeq\langle\ov\chi\rangle\simeq C_2$. This finishes the proof of Theorem~\ref{thm:2}.
\end{proof}

\section{An application}\label{sec4}

We finish with an application to the strong twist conjecture of Jones--Mangioni--Sartori~\cite{JoMaSa1}. We recall only the terminology
needed here. An \emph{Artin system} is a pair $(A,S)$, where $A$ is a group and $S\subset A$ is an Artin generating set, i.e.\ $S$ is identified with the standard generating set of an Artin group; see~\cite[Definition~2.1]{JoMaSa1}. We denote by $\Gamma_S$ the corresponding defining graph. In this section we use the convention of~\cite{JoMaSa1}: two vertices $s,t\in S = V(\Gamma_S)$ are joined with an edge whenever the corresponding Artin exponent $m_{s,t}$ is finite, including the case where $m_{s,t}=2$, while a missing edge means $m_{s,t}=\infty$. This differs from the Coxeter graph convention used earlier in the paper, where edges labeled by $m_{s,t}=2$ are omitted.  For $X\subseteq S$, we denote by $A_X$ the subgroup of $A$ generated by $X$.

The \emph{reflection set} of $(A,S)$ is $R_S=\{gsg^{-1}\mid s\in S,\ g\in A\}$. An \emph{elementary twist} is the operation defined in~\cite[Definition~3.2]{JoMaSa1}: one cuts the defining graph along a separating spherical indecomposable subgraph and conjugates one side by the
corresponding Garside element. Two Artin generating sets are called \emph{twist equivalent} if they are related by a finite sequence of
elementary twists and conjugations. The Artin system $(A,S)$ satisfies the \emph{strong twist conjecture} if every Artin generating set $U$ for $A$ with $R_U=R_S$ is twist equivalent to $S$. A \emph{$1$-chunk} of $\Gamma_S$ is a connected induced subgraph without separating vertices, maximal with these properties; see~\cite[Definition~2.11]{JoMaSa1}.

Jones--Mangioni--Sartori proved the strong twist conjecture for spherical-type Artin systems of type $A_n$, $B_n$ for $n\ge2$, and $D_n$ with
$n\ge4$, $n\ne5$. The case $D_5$ was omitted because the automorphism group of $A[D_5]$ was not known~\cite[Lemma~5.3]{JoMaSa1}. We now fill this gap.

\begin{prop}\label{prop:D5-strong-twist}
Let $(A,S)$ be an Artin system of type $D_5$. Then $(A,S)$ satisfies the strong twist conjecture.
\end{prop}

\begin{proof}
Let $U$ be an Artin generating set for $A$ such that $R_U=R_S$. As in the proof of~\cite[Lemma~5.3]{JoMaSa1}, the equality $R_U=R_S$ implies that $(A,U)$ is again of spherical type and has defining graph of type $D_5$. Hence there exists an automorphism $\psi\in\Aut(A)$ such that $\psi(S)=U$.

Choose an isomorphism from the standard Artin group $A[D_5]$ onto $A$ which sends the standard generating set to $S$. Through this identification, Theorem~\ref{thm:1} applies to $\psi$. Thus $\psi$ is either inner or the composition of an inner automorphism with the inversion automorphism $\chi_S$, where $\chi_S(s)=s^{-1}$ for all $s\in S$.

The second case is impossible. Indeed, if $\psi=\conj_h\circ\chi_S$, then $U=\{hs^{-1}h^{-1}\mid s\in S\}$. Since $R_U=R_S$, every element of $U$ is conjugate to some element of $S$. Hence, for each $s\in S$, the element $s^{-1}$ is conjugate to some $t\in S$. This contradicts the exponent-sum homomorphism $\ell\colon A\to\Z$, defined by $\ell(s)=1$ for all $s\in S$, because $\ell$ is constant on conjugacy classes, whereas $\ell(s^{-1})=-1$ and $\ell(t)=1$. Therefore $\psi$ is inner. Hence $U$ is conjugate to $S$, and in particular $U$ is twist equivalent to~$S$.
\end{proof}

\begin{corl}\label{cor:JMS-D5}
The conclusion of \textup{\cite[Corollary~E]{JoMaSa1}} remains true with type $D_5$ included. More precisely, let $(A,S)$ be an Artin system whose defining graph $\Gamma_S$ is connected. Suppose that, for every $X\subseteq S$ spanning a $1$-chunk in $\Gamma_S$, the Artin system $(A_X,X)$ is one of the types listed in \textup{\cite[Corollary~E]{JoMaSa1}}, with item $(5)$ replaced by spherical type $A_n$ for $n\ge2$, $B_n$ for $n\ge2$, or $D_n$ for $n\ge4$. Then $(A,S)$ satisfies the strong twist conjecture.
\end{corl}

\begin{proof}
The proof is the same as the proof of \cite[Corollary~E]{JoMaSa1}. The only missing input for including type $D_5$ in item $(5)$ was the strong
twist conjecture for Artin systems of type $D_5$, supplied by Proposition~\ref{prop:D5-strong-twist}.
\end{proof}



\end{document}